\documentclass[11pt,righttag]{amsart}

\usepackage{graphicx}

\begin{document}

\topmargin-0.1in
\textheight8.5in
\textwidth5.5in
 
\footskip35pt
\oddsidemargin.5in
\evensidemargin.5in

 \newsymbol\kk 207C       

\newcommand{\V}{{\cal V}}      
\renewcommand{\O}{{\cal O}}
\newcommand{\LL}{\cal L}
\newcommand{\Ext}{\hbox{Ext}}
\newcommand{\pE}{\hbox{$^\pi\kern-2pt E$}}
\newcommand{\hQ}{\hbox{$\hat Q$}}
\newcommand{\phQ}{\hbox{$ '{\hat Q}$}}
\newcommand{\phd}{\hbox{$ '{\hat \delta}$}}

\newcommand{\lonto}{{\protect \longrightarrow\!\!\!\!\!\!\!\!\longrightarrow}}

\newcommand{\m}{{\frak m}}
\newcommand{\gl}{{\frak g}{\frak l}}
\newcommand{\ssl}{{\frak s}{\frak l}}
 
\renewcommand{\d}{\partial}

\newcommand{\ds}{\displaystyle}
\newcommand{\s}{\sigma}
\renewcommand{\l}{\lambda}
\renewcommand{\a}{\alpha}
\renewcommand{\b}{\beta}
\newcommand{\G}{\Gamma}
\newcommand{\g}{\gamma}

\newcommand{\C}{{\Bbb C}}
\newcommand{\N}{{\Bbb N}}
\newcommand{\Z}{{\Bbb Z}}
\newcommand{\ZZ}{{\Bbb Z}}
\newcommand{\K}{{\cal K}}
\newcommand{\ve}{{\varepsilon}}
\newcommand{\cupp}{{\star}}

\newcommand{\rowxy}{(x\ y)}
\newcommand{\colxy}{ \left({\begin{array}{c} x \\ y \end{array}}\right)}
\newcommand{\scolxy}{\left({\begin{smallmatrix} x \\ y
\end{smallmatrix}}\right)}

\renewcommand{\P}{{\Bbb P}}

\newcommand{\la}{\langle}
\newcommand{\ra}{\rangle}

\newtheorem{thm}{Theorem}[section]
\newtheorem{lemma}[thm]{Lemma}
\newtheorem{cor}[thm]{Corollary}
\newtheorem{prop}[thm]{Proposition}

\theoremstyle{definition}
\newtheorem{defn}[thm]{Definition}
\newtheorem{notn}[thm]{Notation}
\newtheorem{ex}[thm]{Example}
\newtheorem{rmk}[thm]{Remark}
\newtheorem{rmks}[thm]{Remarks}
\newtheorem{note}[thm]{Note}
\newtheorem{example}[thm]{Example}
\newtheorem{problem}[thm]{Problem}
\newtheorem{ques}[thm]{Question}
 
\numberwithin{equation}{section}

\newcommand{\onto}{{\protect \rightarrow\!\!\!\!\!\rightarrow}}
\newcommand{\donto}{\put(0,-2){$|$}\put(-1.3,-12){$\downarrow$}{\put(-1.3,-14.5) 

{$\downarrow$}}}

\newcounter{letter}
\renewcommand{\theletter}{\rom{(}\alph{letter}\rom{)}}

\newenvironment{lcase}{\begin{list}{~~~~\theletter} {\usecounter{letter}
\setlength{\labelwidth4ex}{\leftmargin6ex}}}{\end{list}}

\newcounter{rnum}
\renewcommand{\thernum}{\rom{(}\roman{rnum}\rom{)}}

\newenvironment{lnum}{\begin{list}{~~~~\thernum}{\usecounter{rnum}
\setlength{\labelwidth4ex}{\leftmargin6ex}}}{\end{list}}

\title{Quadratic algebras with Ext algebras generated in two degrees}


\author[T. Cassidy]{ }
\maketitle

\begin{center}

\vskip-.2in Thomas Cassidy \\
\bigskip

Department of Mathematics\\ Bucknell University\\
Lewisburg, Pennsylvania  17837
\\ \ \\

\end{center}

\setcounter{page}{1}

\thispagestyle{empty}

 \vspace{0.2in}

\begin{abstract}
\baselineskip15pt
We show that there exist non-Koszul graded algebras that appear to be Koszul up to any given cohomological degree. For any integer $m\ge 3$ we exhibit a non-commutative quadratic algebra for which the corresponding bigraded Yoneda algebra is generated in degrees $(1,1)$ and $(m,m+1)$.  The algebra is therefore not Koszul but is $m$-Koszul (in the sense of Backelin).  These examples answer a question of Green and Marcos \cite{GM}.
 \end{abstract}

\bigskip

\baselineskip18pt

\section{Introduction}
A connected graded algebra $A$  over a field $\kk$ with generators in degree one  is called  Koszul \cite{Priddy} if its associated bigraded Yoneda (or Ext) algebra $E(A)=\bigoplus_{m\le  n} Ext_A^{m,n}(\kk,\kk)$  is generated as an algebra by $Ext_A^{1,1}(\kk,\kk)$.    Koszul algebras are always quadratic, i.e. the elements in a minimal collection of defining relations will always be of degree two, however not every quadratic algebra is Koszul.   
A quadratic algebra $A$ will fail to be Koszul if and only if $Ext_A^{m,n}(\kk,\kk)\ne 0$ for some $m< n$.    

  The notion of $m$-Koszul described in \cite{PP} and credited to Backelin \cite{Bac}   serves as a measure of how close a graded $\kk$-algebra comes to being Koszul.   
The following definition of $m$-Koszul should not be confused with Berger's $N$-Koszul \cite{Berger}, which refers to $N$-homogeneous algebras with Yoneda algebras generated in degrees one and two.

\begin{defn} 
A graded algebra $A$ is called $m$-Koszul if    $\Ext^{ij}_A(\kk,\kk)=0$ for all\/ $i<j\le m$;  
 \end{defn}
While any quadratic algebra is 3-Koszul,  only a  Koszul algebra will be $m$-Koszul for every $m\ge 1$.  Conversely, if $A$   is $m$-Koszul for every $m\ge 1$ then $A$ is Koszul.  It is natural to ask whether a quadratic algebra could be $m$-Koszul for large values of $m$ and yet still fail to be Koszul. 

The purpose of this paper is to show that there exist non-Koszul graded algebras that appear to be Koszul up to any given cohomological degree.  Specifically, we show that for any integer $m\ge3$ there exists an $m$-Koszul algebra $C$ of global dimension $m$ for which the corresponding Yoneda algebra $E(C)$ is generated as an algebra in cohomology degrees one and $m$.     Therefore it is not possible to use a single $m$ to confirm Koszulity by means $m$-Koszulity.   

The algebra $C$  also satisfies the following two conditions:
 \begin{itemize}
    \item[(1)] $\Ext^{m,n}_C(\kk,\kk)=0$ unless $m=\delta(n)$ for a function $\delta:\N\to \N$;
    \item[(2)]  $\Ext_C (\kk,\kk)$
               is finitely generated.
  \end{itemize}
  Such algebras are called $\delta$-Koszul by Green and Marcos \cite{GM}.  In our case the function $\delta$ is given by 
  $$\delta(n)=\biggl\lbrace 
\begin{tabular}{c c}
\it n & \ if $n<m$  \\  
$n+1$  & \ if  $n=m$   \\
\end{tabular}
 $$ 
Our examples answer the third question posed Green and Marcos in \cite{GM}.  They ask if there is a bound $N$ such that if $A$ is a  $\delta$-Koszul algebra, then $E(A)$ is generated in degrees $0$ to $N$.  The algebras $C$  illustrate that   no such bound exists.  Moreover, the bound does not exist even if we restrict ourselves to quadratic algebras.

 A quadratic algebra $A$ is determined by a vector space
of generators $V=A_1$ and an arbitrary subspace of quadratic relations
$I\subset V\otimes V$. 
The free algebra $\kk\langle V\rangle$  carries a standard grading and $A$ inherits a grading from this free algebra.  We denote by $A_n$ the component of $A$ degree $n$.  For any graded algebra $A=\oplus_k A_k$, let $A[j]$ be the same vector space with the shifted grading $A[j]_k=A_{j+k}$. 
Throughout we assume our graded algebras $A$ 
 are locally finite-dimensional with $A_i=0$ for $i<0$ and $A_0=\kk$.

\section{The algebra $C$}

Let $m$ be an greater than 2. If $m=3$ the algebra $C$ has 10 generators and 8 relations.  If $m=4$ then $C$ has 12 generators and 14 relations.  For $n\ge 5$, $C$ has $3m$ generators and $4+3m$ relations.  
     The case  $m=3$ is already well known, and the case $m=4$ will be encompassed in the proof of Lemma \ref{B} as the algebra $B$, so  we will henceforth assume that $m\ge 5$.

 The algebra $C$ is defined as follows.  The generating vector space $V$ has the basis $\bigcup_{i=1}^{m+1}S_i$ with
 sets   $S_1=\{n\},$ $S_2=\{p,q,r\},$ $S_3=\{s,t,u\},$ $S_4=\{v,w,x_1,y_1,z_1\},$ $S_5=\{x_2,y_2,z_2\}, \ \cdots,\ S_{m-1}=\{x_{m-4},y_{m-4},z_{m-4}\},$ $S_{m}=\{x_{m-3},y_{m-3}\},$ and $S_{m+1}=\{x_{m-2}\}$. For all $m\ge 5$ the space of relations $I$ contains the generators $\{np-nq, np-nr, ps-pt, qt-qu, rs-ru, sv-sw, tw-tx_1, uv-ux_1, vx_2,wx_2, x_{i}x_{i+1}, sv-sy_1,tw-ty_1,ux_1-uy_1,sz_1,tz_1,uz_1, y_{i-1}x_{i}+z_{i-1}y_{i}\}$ where $i\le m-3$.  In addition, if   $m\ge 6$ then $I$ also contains $\{z_{i}z_{i+1}\}$ where $i\le m-5$. 
 
 \begin{rmk}
 We have  chosen this large set of generators   to clarify the exactness of the resolution below.  It may be possible to construct examples with fewer generators.
 \end{rmk}
 
 Notice that any basis for  $I$ is formed   from certain sums of elements of $S_i$ right multiplied by elements of $S_{i+1}$.  This ordering on a basis of $I$ makes $C$ highly noncommutative.  Indeed the center of $C$ is just the field $\kk=C_0$.  Moreover, this ordering tells us that the left annihilator of $n$ is zero and more generally that the left annihilator of an element of $S_i$ is generated by sums of elements from $\Pi_{k=j}^{i-1}S_{k}$ for $1\le j\le i-1$.      The structure of $I$ will be exploited in the proofs of Lemmas \ref{globaldim}, \ref{nokilla}, \ref{g'} and \ref{B}.
     
 Our proof relies on constructing an explicit projective resolution for $\kk$ as a left $C$-module.  Let $(P^{\bullet},\lambda)$ be the  sequence of projective $C$-modules
$$   \begin{array}{c c c c c c c c c c c c c c}  P^m&\buildrel {\lambda_m}\over\longrightarrow &P^{m-1}&\buildrel {\lambda_{m-1}}\over\longrightarrow &P^{m-2}&\cdots&    \buildrel {\lambda_2}\over\longrightarrow 
&P^1 & \buildrel {\lambda_1}\over\longrightarrow 
&C   \to \kk  \end{array}$$  
where $P^m=C[-m-1]$, $P^{m-1}=(C[1-m])^7$, $P^{m-2}= (C[2-m])^{16}$, $P^{2}=  (C[-2])^{3m+4}$,   $P^{1}=  (C[-1])^{3m}$, and for $3\le i\le m-3$,  $P^{i}=  (C[-i])^{3m+12-3i}$.

The map from $C$ to $\kk$ is the usual augmentation.  For convenience we will use $\lambda_i$ to denote both the map from $P^i$ to $P^{i-1}$ and the matrix which gives this map via right multiplication.   The map $\lambda_1$ is   right multiplication by the transpose of the matrix 
$$(n \ \ p\ \ q\ \ r\ \ s\ \ t\ \ u\ \ z_1\  \dots \ z_{m-4}\  \ v\  \ w\  \ x_1\  \ y_1\  \ x_2\  \ y_2\  \dots \   x_{m-3}\  \ y_{m-3}\  \ x_{m-2})$$ and the map $\lambda_m$ is   right multiplication by the the matrix 
$$\left({\begin{array}{cccccccccccc}0&0&0&0&np& np& -np&0&\cdots&0 \end{array}}\right).$$  

The remaining maps $\lambda_i$ will be defined as right multiplication by matrices given in block form, for which we will need the following components.  Let 
$$\alpha=  \left({\begin{array}{ccccccccc} 0&0&0&p&0&0&-p&p&0  \\ 0&0&0&0&q&0&0&-q&q \\ 0&0&0&0&0&r&-r&0&r   \end{array}}\right),  \ 
\alpha'=  \left({\begin{array}{cccccc} p&0&0&-p&p&0  \\ 0&q&0&0&-q&q \\ 0&0&r&-r&0&r   \end{array}}\right),$$  \ 
$$\beta=  \left({\begin{array}{cccc} s&-s&0&0  \\ 0&t&-t&0 \\ u&0&-u&0 \\ s&0&0&-s \\ 0&t&0&-t \\ 0&0&u&-u    \end{array}}\right),   
\beta'=  \left({\begin{array}{ccc} s&-s&0  \\ 0&t&-t \\ u&0&-u    \end{array}}\right),  \ 
 \gamma=  \left({\begin{array}{cc} v&0  \\ w&0 \\ x_1&0 \\ y_1& z_1  \end{array}}\right),$$   
 $$   \gamma'=  \left({\begin{array}{c} v  \\ w \\ x_1   \end{array}}\right), \   
   \chi_j=  \left({\begin{array}{cc} x_j&0  \\ y_j& z_j   \end{array}}\right),   
   \delta=  \left({\begin{array}{ccc}  -p&p&0  \\  0&-q&q \\  -r&0&r   \end{array}}\right), \  
   \epsilon=  \left({\begin{array}{c}s\\ t\\ u  \end{array}}\right),$$  
    $$\zeta_j=  \left({\begin{array}{ccccc}  z_1&0&0&\cdots&0  \\  0&z_2&0&\cdots&0 \\ 0&0&z_3&\cdots&0\\ \vdots&\vdots&&\ddots &\vdots \\ 0&0&0&\cdots&z_j    \end{array}}\right),  \ 
    \eta=(0,n,n,-n),\ \ \rm{and}$$  
    $$  \eta'=\left({\begin{array}{cccc}  0&n&-n&0 \\   0&n&0&-n    \end{array}}\right).$$
 
  The matrix defining map $\lambda_2$ has this  block form 
  $$\left({\begin{array}{cccccccccccc}   \eta'\\  & \delta\\  &&\epsilon\\    &&&\zeta_{m-5} \\    &&&&\beta \\  &&&&&\gamma \\  &&&&&&\chi_2 \\&&&&&&&\ddots \\&&&&&&&&\chi_{m-4} \\ &&&&&&&&&x_{m-3}
   \end{array}}\right).$$ 
  
 The matrix   $\lambda_3$ has the form
 $$\left({\begin{array}{cccccccccccc} 0&  \eta\\  && \delta\\  &&&\epsilon\\    &&&&\zeta_{m-6} \\    &&&&&\alpha' \\  &&&&&&\beta \\  &&&&&&&\gamma \\  &&&&&&&&\chi_2 \\&&&&&&&&&\ddots \\&&&&&&&&&&\chi_{m-5} \\ &&&&&&&&&&&x_{m-4}
   \end{array}}\right).$$     
 For $4\le j \le m-3$ the matrix   $\lambda_j$ has the form
 $$\left({\begin{array}{cccccccccccc}   \eta\\  & \delta\\  &&\epsilon\\    &&&\zeta_{m-j-3} \\    &&&&\alpha \\  &&&&&\beta \\  &&&&&&\gamma \\  &&&&&&&\chi_2 \\&&&&&&&&\ddots \\&&&&&&&&&\chi_{m-j-2} \\ &&&&&&&&&&x_{m-j-1}
   \end{array}}\right).$$  
   Note that $\chi_2$ is the only $\chi$ block in $\lambda_{m-4}$ and that the matrix $\lambda_{m-3}$ has no $\zeta$ or $\chi$ blocks.
   
    The matrix   $\lambda_{m-2}$ has the   form $$\left({\begin{array}{ccccccc}   \eta\\  & \delta\\   &&\alpha \\      & &&\beta \\  &&&&\gamma'  
   \end{array}}\right),$$ 
   and the matrix   $\lambda_{m-1}$ has the  form  $$\left({\begin{array}{cccc}   \eta\\  & \alpha\\  &&\beta'\\         
   \end{array}}\right).$$ 

\begin{lemma}\label{globaldim}
Let $(Q^\bullet,\phi)$ be a minimal projective resolution of $_C\kk$ where the map $Q^i\to Q^{i-1}$ is given as right multiplication by a matrix $\phi_i$.  Then the matrices $\phi$ can be chosen to have block form such that all the entries in $\phi_i$ are elements from the subalgebra generated by the set $\cup^{m+2-i}_{j=1} S_j$.
\end{lemma}
\begin{proof} We prove this by induction on $i$.  $\phi_1$ can be chosen to be  $\lambda_1$, which has entries from $V=\cup_{j=1}^{m+1} S_j$.  Since the set $\phi_2\phi_1$ must span the   space  $I$, $\phi_2$ can be chosen to be $\lambda_2$, where the blocks have entries of the appropriate form.    Now suppose   $\phi_i$ has block form with entries from the subalgebra generated by  $\cup^{m+2-i}_{j=1} S_j$. 
Since   the rows of $\phi_{i+1}$ must annihilate the columns of $\phi_i$,   $\phi_{i+1}$ can be chosen to have block form corresponding to the blocks of $\phi_i$.   Recall that any basis for $I$ is ordered so that only elements of $S_j$  appear on the left of elements of $S_{j+1}$.  Since the entries in $\phi_i$ contain no elements from $\cup^{m+1}_{j=m+3-i} S_j$, no elements from $\cup^{m+1}_{j=m+2-i} S_j$ will appear in entries of $\phi_{i+1}$.  Thus the entries in  $\phi_{i+1}$ are from the subalgebra generated by the set $\cup^{m+1-i}_{j=1} S_j$.
 \end{proof}

\begin{rmk} The lemma implies that $\phi_{m+1}$, if it exists, can only contain elements of $S_1$ and that there can be no map $\phi_{m+2}$.  Therefore a minimal resolution of $_{C}\kk$ would have length  no more than $m+1$ and so the global dimension of $C$ is at most $m+1$. We will see later that the global dimension of $C$ is exactly $m$.
\end{rmk}

\begin{lemma}\label{nokilla}
The left annihilators of $\eta$, $\eta'$, $\lambda_m$ and $\alpha$ are zero.
\end{lemma}
\begin{proof}  The relations for $C$ make it clear that nothing annihilates $n$ from the left, and consequently $\eta$, $\eta'$ and $\lambda_m$ cannot be annihilated from the left.  Since the entries in $\alpha$ are all from $S_2$, the annihilator would have to be made from left multiples of $n$.  However  $p$, $q$ and $r$ each appear alone in the first columns of $\alpha$ and $n$ does not annihilate these individually.
\end{proof}

\begin{lemma}\label{g'} 
The rows of $\gamma'$ generate the left annihilator of $x_2$.
\end{lemma}
\begin{proof}  The left annihilator of $x_2$ can have generators made from sums of elements in $S_4$, $S_3S_4$, $S_2S_3S_4$ or $S_1S_2S_3S_4$.  It is easy to check that linear combinations of $v$, $w$ and $x_1$ are the only elements in the span of $S_4$ that annihilate $x_2$.  The elements of $S_3S_4$ span a six dimensional subspace of $C_2$ with basis $\{sv, sx_1, tv, tw, uv, uw\}$ and these are all left multiples of $v$, $w$ or $x_1$.  It follows that in $C$ the elements of  $S_2S_3S_4$ and $S_1S_2S_3S_4$ are also all left multiples of $v$, $w$ or $x_1$.
\end{proof}

\begin{lemma}\label{B}
The left annihilator of $\beta$ is generated by the rows of $\alpha$, the left annihilator of $\delta$ is generated by $\eta$,  the left annihilator of $\beta'$ is generated by $\lambda_{m+1}$, and  the left annihilator of $\gamma'$ is generated by the rows of $\beta'$.
\end{lemma}
\begin{proof}   We consider an algebra  $B$ closely related to $C$.     
Let $B$ be the algebra with generators $\{n,p,q,r,s,t,u,v,w,x_1, y_1, a, b\}$ and defining relations $\{np-nq, np-nr, ps-pt, qt-qu, rs-ru, sv-sw, tw-tx_1, uv-ux_1, va-vb, wa-wb, x_1a-x_1b\}$.  We will resolve $\kk$ as  a  $B$-module using maps made from the same blocks as those which appear in the resolution of $\kk$ as a $C$ module.

 Let $(R^{\bullet},\psi)$ be this  sequence of projective $B$-modules
$$\begin{array}{ccc c c c cc c c c c c} &R^4&&R^3&&R^2&&R^1&&R^0\\
0  \rightarrow &B[-5] &\buildrel {\psi_{4}}\over\longrightarrow  & B[-3]^7&\buildrel {\psi_{3}}\over\longrightarrow &B[-2]^{14}&   \buildrel {\psi_2}\over\longrightarrow 
&B[-1]^{13} & \buildrel {\psi_1}\over\longrightarrow 
&B  & \to \kk  \end{array}$$  
The map from $B$ to $\kk$ is the usual augmentation,  the map $\psi_1$ is  right multiplication by the transpose of the matrix $$\left({\begin{array}{ccccccccccccc}n & p & q & r &s& t & u &  v & w & x_1& a & b\end{array}}\right),$$   
the map $\psi_2$ is   right multiplication by the the matrix $$\left({\begin{array}{ccccccc}   \eta'\\  & \delta\\   &&\beta\\&&&\gamma'&-\gamma' 
   \end{array}}\right),$$
   the map $\psi_3$ is   right multiplication by the the matrix $$\left({\begin{array}{ccccc}   0&\eta\\  & &\alpha\\  & &&\beta'   \end{array}}\right),$$
and the map $\psi_4$ is  right multiplication by the the matrix $$\left({\begin{array}{ccccccc}  0&0&0&0&np&np&-np
   \end{array}}\right).$$
Since the product $\psi_i\psi_{i-1}$ is zero in $B$, $R^\bullet$ is a complex.  The list of generators and relations for $B$ ensure that  this complex is exact at $R^1$ and $R^0$.     Our goal is to show that $(R^\bullet,\psi)$ is a minimal projective resolution of $_B\kk$.
 
 We observe that $B$ is the associated graded algebra for the splitting algebra (see \cite{RSW3}) corresponding to the layered graph below.

   \includegraphics[height=13cm]{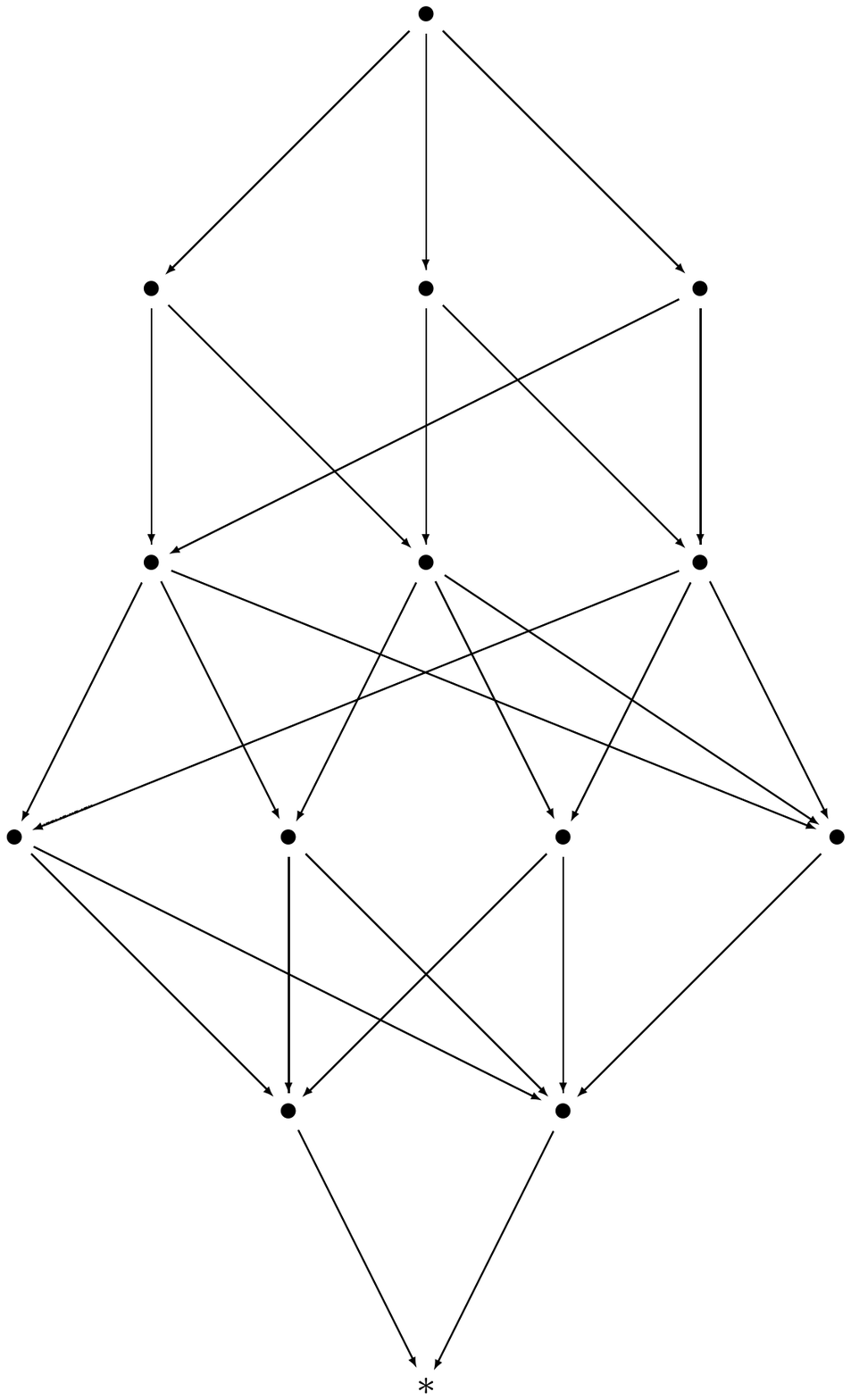} 
 \vskip-1cm
The methods of \cite{RSW3} (see also \cite{RSW1}) show that $B$ has Hilbert series $H_B(g)=(1-13g+14g^2-7g^3+g^5)^{-1}$.   
Let $P_B(f,g)=\ds\sum_{i,j=0}^\infty dim(Ext_B^{ij}(\kk,\kk)f^ig^j$ be the double Poincar\'e series for $B$.   From the formula $P_B(-1,g)=H_B^{-1}(g)=1-13g+14g^2-7g^3+g^5$  we can deduce  something about the shape of a minimal projective resolution of the $B$-module $\kk$.  Moreover, the entries in the maps of a minimal  projective resolution come from certain sets  as in   Lemma \ref{globaldim}.  It follows that the resolution must have the form:  
$$   \begin{array}{c c c c c c c c c c c c c } 0&\to & B[-5]^{d_3-d_2}  \longrightarrow &R^4\oplus B[-5]^{d_3}\oplus B[-4]^{d_1} & \longrightarrow  &  \end{array}$$
$$   \begin{array}{c c c c c c c c c c c c c }
R^3\oplus B[-4]^{d_1}\oplus B[-5]^{d_2}& \longrightarrow &R^2&   \buildrel {\psi_2}\over\longrightarrow  
&R^1 & \buildrel {\psi_1}\over\longrightarrow 
&R^0  & \to \kk  \end{array}$$  
We will show that $d_1=d_2=d_3=0$ so that $R^\bullet$ is in fact a resolution
of $_B\kk$. 

Consider first the possibility that $d_1>0$.  This can only happen if $\eta$, $\alpha$ or $\beta'$ are annihilated from the left by a linear vector.  Clearly nothing annihilates $\eta$ from the left and by Lemma \ref{nokilla} nothing  annihilates $\alpha$ from the left.  Suppose the vector  
 $$\left({\begin{array}{ccccc} e_1& e_2& e_3 \end{array}}\right)$$  
left annihilates $\beta'$ where each $e_i$ is a sum of elements from $S_2$.  Then 
 $$\vec{e}=\left({\begin{array}{cccccccc} e_1& e_2& e_3&0&0&0 \end{array}}\right)$$ 
would left annihilate  the $\beta$ which appears in $\psi_2$.  However the Poincar\'e series for $B$ assures us that the dimension of $Ext_B^{3,3}(\kk,\kk)$ is seven, which means that $\vec{e}$ would be in the span of the rows of $\alpha'$.  Now observe that no nonzero combination of the rows of $\alpha'$ could produce $\vec{e}$.  We conclude that $d_1=0$.

Since $d_1$ is zero, $Ext^{3,4}(\kk,\kk)=Ext^{4,4}(\kk,\kk)=0$.  Now observe that  the absence of $Ext^{4,4}(\kk,\kk)$ means that $Ext^{5,5}(\kk,\kk)$ must also be zero, so that $d_2=d_3$.  For $Ext^{3,5}(\kk,\kk)$ to be nonzero, the matrix defining the map $R^3 \oplus B[-5]^{d_2}  \to R^2$ must contain sums of elements from $S_1S_2S_3$ which annihilate $\gamma'$.  The elements of $S_1S_2S_3$ span a one dimensional subspace of $C_3$ with basis $\{nps\}$.  Suppose $nps(\ h_1\ \ h_2\ \ h_3\ )\gamma'=$ for some $h_i\in\kk$.  Since  in $C_4$  $npsv=npsw=npsx_1$, we get $h_1+h_2+h_3=0$, and thus $nps(\ h_1\ \ h_2\ \ h_3\ )$ is just   a row of $\beta'$ multiplied on the left by $np$.  Therefore the rows of $\beta'$ generate  the left annihilator of $\gamma'$ and $d_3=d_2=0$, which means $(R^\bullet,\psi)$ is exact.

  In general one would not expect information from resolutions over other algebras to be useful in resolving $_{C}\kk$, however  the   relations of $C$ and $B$ both follow the pattern that the left annihilator of an element of $S_i$ is generated by elements of $S_{i-1}$, so that information from $B$ is applicable to $C$.     
   Thus from the fact that $(R^\bullet,\psi)$ is exact, we see that the left annihilator of $\beta$ is generated by the rows of $\alpha$, the left annihilator of $\delta$ is generated by $\eta$,  the left annihilator of $\beta'$ is generated by $\lambda_{m+1}$, and  the left annihilator of $\gamma'$ is generated by the rows of $\beta'$.  
\end{proof}

\begin{thm} \label{main}  For any integer $m\ge 3$ the complex $P^\bullet$ is a minimal projective resolution of the left $C$-module $\kk$.  It follows that the algebra $C$ has global dimension $m$ and   $\Ext^{ij}_{C}(\kk,\kk)=0$ for all\/ $i<j\le m$.  Moreover $C$ is not a  Koszul algebra because   $\Ext^{m,m+1}_{C}(\kk,\kk)\ne 0$.   
\end{thm}

\begin{proof} 
Direct calculation shows that $\lambda_i\lambda_{i-1}=0$ for all $i$ so that $P^\bullet$ is a complex.  It is clear from the block form of the matrices $\lambda_i$  that their rows are linearly independent.     Since nothing annihilates $n$ from the left, it is clear that    $P^\bullet$ is exact at $P^m$ and that none of the $\lambda_i$ needs an additional row to annihilate $\eta$ or $\eta'$.
 The complex is exact at $P^1$ since the product $\lambda_2\lambda_1$ gives the defining relations for $C$.   We will show that $P^\bullet$ is exact elsewhere by examining the component blocks in the matrices $\lambda_i$.    

The block  $\epsilon$ appears in   $\lambda_1$ and is annihilated on the left by the $\delta$ which appears in $\lambda_2$.  For $i>1$,  the column of $\lambda_i$ containing $\epsilon$ has no other nonzero entries.  Suppose that for some $d_i\in C$ we have $(\ d_1 \ \ d_2\ \ d_3\ )\ \epsilon=0$.  Since $\lambda_2\lambda_1$ give a basis for  $I$, it follows that 
$$\left({\begin{array}{ccccccccccc}0&0&0&0&d_1&d_2&d_3&0&\cdots&0 \end{array}}\right)$$
is a sum of left multiples of the rows of $\lambda_2$.  By the block form of $\lambda_2$ this means that  $(\ d_1 \ \ d_2\ \ d_3\ )$ is a sum of left multiples of the rows of $\delta$.  Therefore the rows of $\delta$ generate the left annihilator of $\epsilon$. 
In the same manner, one sees that  the rows of $\epsilon$ generate the left annihilator of $z_1$ and that $z_i$  generates the left annihilator of $z_{i+1}$.
 
   While $\gamma$ does not appear in $\lambda_1$, the matrix $(v,w,x_1,y_1)^t$ does, and is annihilated by $\beta$.  Any row annihilating $\gamma$ must also annihilate $(v,w,x_1,y_1)^t$, and hence the rows of $\beta$ generate the left annihilator of $\gamma$.  Likewise, since $(x_i,y_i)^t$ appears in $\lambda_1$ we see that the rows of $\gamma$ generate the left annihilator of $\chi_2$ and that the rows of $\chi_i$ generate the left annihilator of $\chi_{i+1}$.  For $i>2$, when $x_{i}$ appears as the only nonzero entry in a column of $\lambda_{m-1-i}$, it can be annihilated by at most the rows of $\chi_{i-1}$ since  $\chi_{i-1}$ annihilates the $(x_i, y_i)^t$ in $\lambda_1$.   
Since the second row of $\chi_{i-1}$ does not annihilate $x_i$, we conclude that $x_{i-1}$ generates the left annihilator of $x_i$.

 Lemmas \ref{nokilla}, \ref{g'} and \ref{B} complete the proof that $P^\bullet$ is exact.  Since $C$ is a graded algebra and the maps $\lambda_i$ are all of degree at least one, this resolution is minimal.
  \end{proof}   

\bibliographystyle{amsplain}
\bibliography{bibliog}

\end{document}